\documentclass[12pt]{amsart}
\usepackage{amsfonts,amsmath,amscd,amssymb,paralist,amsthm}
\usepackage{mathrsfs}
\usepackage{amsxtra}
\usepackage[mathscr]{eucal}
\usepackage{graphicx}
\usepackage{latexsym,graphicx,verbatim,nameref}
\usepackage[numbers,sort&compress]{natbib}
\usepackage[all,cmtip]{xy}
\usepackage{hyperref}
\usepackage{fancyhdr}
\usepackage{color}
\usepackage{tikz}
\usepackage{tikz-cd}
\usetikzlibrary{arrows,scopes,matrix}
\numberwithin{equation}{section}

\theoremstyle{definition}

\makeatletter
\newtheorem*{rep@theorem}{\rep@title}
\newcommand{\newreptheorem}[2]{%
\newenvironment{rep#1}[1]{%
 \def\rep@title{#2 \ref{##1}}%
 \begin{rep@theorem}}%
 {\end{rep@theorem}}}
\makeatother

\newreptheorem{theorem}{Theorem}
\newreptheorem{corollary}{Corollary}
\newreptheorem{lemma}{Lemma}

\newtheorem{theorem}{Theorem}[section]

\newtheorem{lemma}[theorem]{Lemma}

\newtheorem*{theorem*}{Theorem}
\newtheorem*{proposition*}{Proposition}

\newtheorem{definition}[theorem]{Definition}

\newtheorem{fact}[theorem]{Fact}

\newtheorem*{claim*}{Claim}

\newtheorem*{conjecture*}{Conjecture}
\newtheorem*{observation*}{Observation}
\newtheorem*{question*}{Question}

\begin{document}

\title{Invariant Radon measures and minimal sets for subgroups of $\text{Homeo}_+(\mathbb{R})$ }

\date{\today}

\author{Hui Xu \& Enhui Shi \& Yiruo Wang}

\begin{abstract}
 Let $G$ be a subgroup of $\text{Homeo}_+(\mathbb{R})$ without crossed elements. We show the equivalence among three items: (1) existence of $G$-invariant Radon measures
 on $\mathbb R$; (2) existence of  minimal closed subsets of $\mathbb R$; (3)  nonexistence of infinite towers covering the whole line. For a nilpotent subgroup $G$ of $\text{Homeo}_+(\mathbb{R})$, we show that $G$ always has an invariant Radon measure and a minimal closed set if every element of $G$ is $C^{1+\alpha} (\alpha>0$); a counterexample of $C^1$ commutative subgroup of $\text{Homeo}_+(\mathbb{R})$ is constructed.
\end{abstract}
\maketitle

\section{Introduction}

In the theory of dynamical systems, the following two facts are well known:

(1) if $G$ is a group consisting of homeomorphisms on a compact metric space $X$, then $G$ has
a minimal set $K$ in $X$, that is $K$ is minimal among all nonempty $G$-invariant closed subsets with respect to the inclusion relation on sets;

(2) if $G$ is an amenable group consisting of homeomorphisms  on a compact metric space $X$, then $G$ has an invariant Borel probability  measure on $X$.

In general, these two results do not hold if $X$ is not compact. However, if the topology of $X$ is very constrained and the acting group $G$ possesses some specified structures, then the
existence of invariant Radon measures (Borel measures which are finite on every compact set) or minimal sets can still be true, even if $X$ is noncompact.

When $X$ is the real line $\mathbb R$ and $\Gamma$ is a finitely generated virtually nilpotent group, Plante obtained the following theorem in \cite{Plante}.

\begin{theorem}\label{plante theorem}
If $\Gamma$ is a finitely generated virtually nilpotent subgroup of $\text{Homeo}_+(\mathbb{R})$, then $\Gamma$ preserves a Radon measure on the line.
\end{theorem}

 Here, $\text{Homeo}_+(\mathbb{R})$ means the orientation preserving homeomorphism group on $\mathbb R$.
The following theorem appears in A. Navas' book (see Prop. 2.1.12 in \cite{Na1}).

\begin{theorem}\label{navas theorem}
Every finitely generated subgroup of $\text{Homeo}_+(\mathbb{R})$ admits a nonempty minimal invariant closed set.
\end{theorem}

In this paper, we are interested in non-finitely generated subgroups of $\text{Homeo}_+(\mathbb{R})$
and get the following theorem.

\begin{theorem}\label{main theorem 1}
Let $G$ be a subgroup of $\rm{Homeo}_{+}(\mathbb{R})$ without crossed elements. Then the following items are equivalent:
\begin{enumerate}
  \item [(1)] there exists a $G$-invariant Radon measure;
  \item [(2)] there exists a nonempty closed minimal set;
  \item [(3)] there does not exist any infinite tower $\{(I_i, f_i)\}_{i=1}^\infty$ such that
  $\bigcup_{i=1}^\infty I_i=\mathbb{R}$.
\end{enumerate}
\end{theorem}

Note that the condition having crossed elements implies the existence of free sub-semigroup (\cite[Lemma 2.2.44]{Na1}). In particular, a nilpotent group has no crossed elements, since it contains no free sub-semigroup. For a nilpotent group $G$ of $\text{Homeo}_+(\mathbb{R})$, if it does not preserve any Radon measures, we can construct a better infinite tower. This together with a beautiful generalization of Kopell's Lemma due to A. Navas implies the following theorem.
\begin{theorem}\label{main corollary}
For every $\alpha>0$, every nilpotent $C^{1+\alpha}$  subgroup of $\text{Homeo}_+(\mathbb{R})$
has an invariant Radon measure and has a minimal closed invariant set.
\end{theorem}

We should note that there is no requirement of finite generation or even countability for the group appearing in Theorem \ref{main theorem 1} and
Theorem \ref{main corollary}. This is the key point that differs from Theorem \ref{plante theorem} and Theorem \ref{navas theorem}.

As a supplement of Theorem \ref{main corollary}, we construct in Section 5 a $C^1$ commutative subgroup of  $\text{Homeo}_+(\mathbb{R})$,
which has neither invariant Radon measure nor minimal closed set.

\section{Notions and Auxiliary Lemmas}
In this section, we give some definitions and lemmas which will be used in the proof of the main theorems.\\

Let $G$ be a subgroup of $\text{Homeo}_+(\mathbb{R})$. For $x\in \mathbb{R}$, we denote the {\it orbit} of $x$ by $Gx\equiv\{g(x):g\in G\}$. For $g\in G$, we denote by $\text{Fix}(g)$ the set of fixed points of $g$ and denote by $\text{Fix}(G)$ the set of global fixed points of $G$, i.e. ${\rm{Fix}}(G)=\{x\in\mathbb{R}:\forall g\in G, g(x)=x\}$.

\begin{definition}\label{cross}
Tow elements $f,g\in \text{Homeo}_+(\mathbb{R})$ are called {\it crossed} if there exists an interval $(a,b)$ such that one of $f,g$, saying $f$, $\text{Fix}(f)\cap [a,b]=\{a,b\}$ while $g$ sends either $a$ or $b$ into $(a,b)$. Here we allow the cases $a=-\infty$ or $b=+\infty$.
\end{definition}

 We recall the following  facts which will be used frequently.

\begin{fact}\label{fact1}
 Let $G$ be a subgroup of $\text{Homeo}_+(\mathbb{R})$ and $F=\{f\in G: \text{Fix}(f)\neq \emptyset\}$. Then for any $f\in F$ and $g\in G$, $gfg^{-1}\in F$.
\end{fact}
\begin{proof}
For any $x\in \text{Fix}(f)$, $(gfg^{-1})(g(x))=g(f(x))=g(x)$.
Thus $g(x)\in \text{Fix}(gfg^{-1})$ and hence $gfg^{-1}\in F$.
\end{proof}

\begin{fact}\label{fact2}
 Let $f$ and $g$ be in $\text{Homeo}_+(\mathbb{R})$ which are not crossed and $\text{Fix}(f)\cap [\alpha,\beta]=\{\alpha,\beta\}$. If $g(\alpha)>\alpha$ or $g(\beta)<\beta$, then $g((\alpha, \beta))\cap(\alpha, \beta)=\emptyset$.
 \end{fact}

\begin{fact}\label{fact3}
Let $G$ be a subgroup of $\text{Homeo}_+(\mathbb{R})$. For any $x\in\mathbb{R}$, set
  \begin{displaymath}
  \alpha:=\inf\{Gx\},~~~\beta:=\sup\{Gx\}.
  \end{displaymath}
  Then either $\alpha=-\infty$ (resp. $\beta=+\infty$) or $\alpha\in \rm{Fix}(G)$ (resp. $\beta\in \rm{Fix}(G)$).
\end{fact}

\begin{proof}
We may assume that $\alpha\neq -\infty$. Then for any $g\in G$,
\begin{displaymath}
g(\alpha)\geq \alpha,~~\text{and}~~g^{-1}(\alpha)\geq \alpha\Longrightarrow g(\alpha)\leq \alpha .
\end{displaymath}
Hence $g(\alpha)=\alpha$. It is similar for $\beta$.
\end{proof}

\begin{definition}
If $\{I_i\}_{i=1}^\infty$ is a sequence of closed intervals such that $I_1\subsetneq I_2\subsetneq ...$, and $\{f_i\}_{i=1}^\infty$ is
a sequence of orientation preserving homeomorphisms on $\mathbb R$ such that ${\rm Fix}(f_i)\cap I_i={\rm End}(I_i)$ for each $i$, where ${\rm End}(I_i)$ denotes the endpoint set of interval $I_i$, then
we call the sequence of pairs $\{(I_i, f_i)\}_{i=1}^\infty$ an {\it infinite tower}.
\end{definition}

\begin{lemma}\label{infinity tower}
Let $H$ be a subgroup of $\text{Homeo}_+(\mathbb{R})$ without crossed elements.
Suppose $H$ has no infinite tower $\{(I_i, f_i)\}_{i=1}^\infty$ such that $\bigcup_{i=1}^{\infty}I_i=\mathbb{R}$.
If for every $f\in H$, $\text{Fix}(f)\neq \emptyset$,  then $\text{Fix}(H)\neq\emptyset$.
\end{lemma}

\begin{proof}
Assume to the contrary that $\text{Fix}(H)=\emptyset$. Choose an $f_1\neq \text{id}\in H$.  
 Then $\emptyset\not=\text{Fix}(f_1)\subsetneq\mathbb{R}$.  Take a connected component $(\alpha_1, \beta_1)$ of  $\mathbb{R}\setminus \text{Fix}(f_1)$.

We claim that $-\infty<\alpha_1<\beta_1<+\infty$. In fact, since $\text{Fix}(f_1)\neq\emptyset$,  at least one of $\alpha_1,\beta_1$ is finite. We may assume that $\alpha_1\in\mathbb{R}$. Since $\text{Fix}(H)=\emptyset$, by Fact \ref{fact3}, there exists $f_2\in H\setminus\{f_1\}$ such that
$f_2(\alpha_1)>\max\{\alpha_1, 2\}$. Since $f_1$ and $f_2$ are not crossed,  $f_2((\alpha_1, \beta_1))\cap (\alpha_1, \beta_1)=\emptyset$ by Fact \ref{fact2}.  Therefore, $\beta_1\leq f_2(\alpha_1)<+\infty$.

Set $\alpha_2=\inf\{f_2^i(\alpha_1): i\in\mathbb Z\}$ and $\beta_2=\sup\{f_2^i(\alpha_1): i\in\mathbb Z\}$. Then either $\alpha_2\not=-\infty$ or
$\beta_2\not=+\infty$ by the assumption that $\text{Fix}(f)\neq \emptyset$ for every $f\in H$. Similar to the argument of the previous claim,
we have $\alpha_2\in\mathbb R$ and $\beta_2 \in \mathbb R$. Then $\alpha_2<\alpha_1<\beta_1<\beta_2$ and $\text{Fix}(f_2)\cap [\alpha_2,\beta_2]=\{\alpha_2,\beta_2\}$ and $\beta_2>2$.

Similar to the above arguments, we get $\alpha_3, \beta_3 \in \mathbb R$ and $f_3\in H$ such that $\alpha_3<\alpha_2<\beta_2<\beta_3$, and $\text{Fix}(f_3)\cap [\alpha_3,\beta_3]=\{\alpha_3,\beta_3\}$, and $\alpha_3<-3$.

 Continuing this process, we obtain a nested closed intervals $[\alpha_1,\beta_1]\subsetneq [\alpha_2,\beta_2]\subsetneq\cdots$ and a sequence $f_1,f_2,\cdots\in H$ such that
\begin{displaymath}
\text{Fix}(f_i)\cap [\alpha_i,\beta_i]=\{\alpha_i,\beta_i\}, i=1,2,...,
\end{displaymath}
and $\alpha_{2i-1}<-(2i-1)$ and $\beta_{2i}>2i$ for each $i>0$.
Set $I_i= [\alpha_i,\beta_i]$. Then $\{(I_i, f_i)\}_{i=1}^\infty$ is an infinite tower such that $\bigcup_{i=1}^{\infty}I_i=\mathbb{R}$, which contradicts the hypothesis.
\end{proof}

\begin{lemma}\label{generator fix}
Let $F$ be a subset of $\text{Homeo}_+(\mathbb{R})$ and let $H=\langle F\rangle$ be the group generated by $F$. Suppose $H$ has no crossed elements.
If for every $f\in F$, $\text{Fix}(f)\neq \emptyset$,  then $\text{Fix}(g)\neq\emptyset$ for every $g\in H$.
\end{lemma}

\begin{proof}
Since $H$ is generated by $F$, we need only to prove that for any $g_1\not=g_2\in H$, if $\text{Fix}(g_1)\not=\emptyset$ and $\text{Fix}(g_2)\not=\emptyset$,
then $\text{Fix}(g_1g_2)\not=\emptyset$. Otherwise, $\text{Fix}(g_2)\subset\mathbb R\setminus \text{Fix}(g_1)$ and
$\text{Fix}(g_1)\subset\mathbb R\setminus \text{Fix}(g_2)$. This clearly implies the existence of crossed elements in $H$, which is a contradiction.
\end{proof}

Recall that a subgroup $H$ of $\text{Homeo}_+(\mathbb{R})$ is said to act on $\mathbb R$ {\it freely}, if every non-identity element of $H$ has no fixed points.

\begin{lemma}[H\"{o}lder \cite{Na1} Proposition 2.2.29]\label{Holder}
 Every group acting freely by homeomorphisms of the real line is isomorphic to a subgroup of $(\mathbb{R},+)$.
\end{lemma}

\begin{lemma}\label{uncountable case}
Let $G$ be a subgroup of $\rm{Homeo}_{+}(\mathbb{R})$ and let $\Gamma=\{f\in G: \text{Fix}(f)\neq\emptyset\}$. Suppose $\Gamma$ is a normal subgroup of $G$. If $\text{Fix}(\Gamma)$ is uncountable, then there exist a $G$-invariant Radon measure on ${\mathbb R}$ and a nonempty minimal closed subset of $\mathbb R$.
\end{lemma}

\begin{proof}
If $G=\Gamma$, then each point $x$ in $\text{Fix}(\Gamma)$ is minimal and the Dirac measure $\delta_x$ is a $G$-invariant Radon measure on $\mathbb R$.
So, we may suppose that $\Gamma$ is a proper subgroup of $G$.

Let $\varphi$ be the map on $\mathbb R$ defined by collapsing the closure of each  component of $\mathbb{R}\setminus \text{Fix}(\Gamma)$ into a point.
Then the space $\varphi(\mathbb R)$  is homeomorphic to an interval $K$ (with or without endpoints).
Since $\Gamma$ is normal in $G$, $g(\text {Fix}(\Gamma))=\text{Fix}(\Gamma)$ for every $g\in G$. Thus $G/\Gamma$ naturally acts on $K$
by letting $g\Gamma.\varphi(x)=\varphi (g(x))$.

If $p$ is an end point of $K$, then $p$ is  $G/\Gamma$-invariant and hence $\varphi^{-1}(p)$ is a
connected $G$-invariant closed set $J$ in $\mathbb R$. Let $q$ be a boundary point of
$J$. Then $q$ is $G$-fixed, which contradicts the assumption that $\Gamma$ is properly contained in $G$.
So, $\varphi(\mathbb R)\cong\mathbb R$.

We claim that the action of $G/\Gamma$ on $\varphi(\mathbb R)$ is free. Otherwise, there is some $g\in G\setminus\Gamma$ and
$y\in \varphi(\mathbb R)$ such that $g\Gamma(y)=y$. Then $\varphi^{-1}(y)$ is a $g$-invariant closed interval, and hence
each point of the  boundary of $\varphi^{-1}(y)$  is $g$-fixed. This is a contradiction.

By the claim and H\"{o}lder's Lemma \ref{Holder}, this $G/\Gamma$ action on $\varphi(\mathbb R)$ is conjugate to translations on the line.
We may as well assume that $G/\Gamma$ are translations on $\mathbb R$. Then the Lebesgue measure $\lambda$ on $\mathbb R$ is a
$G/\Gamma$-invariant Radon measure.

Since $\varphi$ is increasing and continuous, it is well known that there is a unique continuous Radon measure
$\ell$ on $\mathbb R$ such that
\begin{equation*}
\ell([a,b])=\varphi(b)-\varphi(a)=\lambda (\varphi([a,b])).
\end{equation*}
 The $G$-invariance of $\ell$ can be seen from
 \begin{equation*}
\ell(g[a,b])=\lambda (g\Gamma\varphi([a,b]))=\lambda(\varphi([a,b]))=\ell([a,b]).
\end{equation*}
Then $\ell$ is the required Radon measure on $\mathbb R$.\\

To prove the existence of minimal sets, we discuss in two cases.\\

\noindent{\bf Case 1.} The $G/\Gamma$-action on $\mathbb R$ is minimal.\\
\indent Set $K=\mathbb{R}\setminus\bigcup_{x\in\mathbb{R}}{\rm{int}}(\varphi^{-1}(x))$. Firstly, $K$ is nonempty, since $\varphi$ is monotonic and $\varphi(\mathbb{R})\cong\mathbb{R}$. Furthermore, for any $x\in K$, $\varphi^{-1}(\varphi(x))$ has at most two points. We claim that $K$ is a minimal closed subset for $G$. For any $x, y\in K$, by the minimality of the $G/\Gamma$-action, there exists a sequence $(g_n)_{n=1}^{\infty}$ in $G$ such that
\[ g_n\Gamma\cdot\varphi(x)=\varphi(g_nx)\rightarrow \varphi(y),~\text{as } n\rightarrow\infty.\]  If $\varphi^{-1}(\varphi(y))=\{y\}$, then there is a subsequence of $(g_{n_k})_{k=1}^{\infty}$ such that $g_{n_k}(x)\rightarrow y $ as $k\rightarrow\infty$. If $\varphi^{-1}(\varphi(y))=\{y,y'\}$, then we may assume that $y<y'$ and that the choice of $(g_n)$ satisfying that $\varphi(g_nx)$ tends to $\varphi(y)$ from left. Then there is also a subsequence of $(g_{n_k})_{k=1}^{\infty}$ such that $g_{n_k}(x)\rightarrow y $ as $k\rightarrow\infty$. Therefore, $K$ is a nonempty minimal closed subset for $G$.\\

\noindent{\bf Case 2.} The $G/\Gamma$-action on $\mathbb R$ is not minimal. \\
\indent Noting that the action of $G/\Gamma$ on $\mathbb R$ consists of translations, $\Lambda\equiv \{g\Gamma(0): g\in G\}$ is discrete and minimal in this case. Take $x\in \varphi^{-1}(\Lambda)\cap \text{Fix}(\Gamma)$ and let $E=\overline{Gx}$. For any $y\in E$, there is an $\epsilon>0$ such that if $d(gx, y)<\epsilon$ then
$\varphi(gx)=\varphi(y)$ by the discreteness of $\Lambda$. Supposed that there exist $g,g'\in G$ such that $\varphi(gx)=\varphi(g'x)$. Then $g\Gamma \cdot \varphi(x)=g'\Gamma \cdot\varphi(x)$. Hence $g\Gamma=g'\Gamma$, by the freeness of the $G/\Gamma$-action. Thus $g(x)=g'(x)$, since $x\in{\rm{Fix}}(\Gamma)$. Therefore, $Gx\cap \varphi^{-1}(z)$ has at most one point, for every $z\in\mathbb{R}$.
Thus we have $gx=g'x$ for any $g, g'\in G$ with $d(gx, y)<\epsilon$ and $d(g'x, y)<\epsilon$. This forces $y=g(x)$ for some $g\in G$. Thus $E=Gx$ is
only a single orbit, which is clearly a nonempty minimal closed subset.
\end{proof}

\section{Proof of Theorem \ref{main theorem 1}}
In this section, we will prove Theorem \ref{main theorem 1}. We prove the theorem by showing $(1)\Longleftrightarrow (3)$ and $(2)\Longleftrightarrow (3)$.

\begin{claim*}[$(1)\Longrightarrow (3)$]
{\it For any subgroup $G$ of  $\rm{Homeo}_{+}(\mathbb{R})$ without crossed elements, if there exists a $G$-invariant Radon measure, then there does not exist an infinite tower covering the line.}
\end{claim*}
 Let $\mu$ be a $G$ invariant Radon measure on $\mathbb R$. If there exists an infinite tower $\{(I_i, f_i)\}_{i=1}^\infty$ such that $\bigcup_{n=1}^\infty I_n=\mathbb{R}$, then there is $N\in\mathbb{N^+}$ such that $\mu (\text{int}(I_{N}))>0$.  Let $B=\text{int}(I_{N})$. By the definition of infinite tower and Fact 2.3, we see that $B, f_{N+1}(B),f_{N+1}^2(B),... $ are pairwise disjoint and are all contained in $I_{N+1}$. Since $\mu$ is $G$-invariant, we have $$\mu(B)=\mu(f_{N+1}(B))=\mu(f_{N+1}^2(B))=\cdots$$
 and then
\begin{displaymath}
\mu(I_{N+1})\geq \sum_{i=0}^\infty \mu(f_{N+1}^i(B))=\infty,
\end{displaymath}
which contradicts the assumption that $\mu$ is a Radon measure.\\

\begin{claim*}[$(2)\Longrightarrow (3)$]
{\it For any subgroup $G$ of  $\rm{Homeo}_{+}(\mathbb{R})$ without crossed elments, if there exists a nonempty minimal closed subset, then there does not exist an infinite tower covering the line.}
\end{claim*}

Assume that $\Lambda$ is a nonempty closed minimal subset of $\mathbb{R}$. Fix a point $x\in \Lambda$. If  there exists an infinite tower $\{(I_i, f_i)\}_{i=1}^\infty$  such that $\bigcup_{n=1}^\infty I_n=\mathbb{R}$, then there exists $N\in \mathbb{N}^+$ such that $x\in \text{int}(I_N)$. Write $I_N=[a,b]$. We may assume that $f_N(x)>x$, otherwise replace $f_N$ by $f_{N}^{-1}$. Then $\lim_{n\rightarrow+\infty} f_N^n(x)=b$. Then $b\in \text{Fix}(f_N)\cap \Lambda$. Since $\Lambda$ is minimal, there must be some $g\in G$ sending $b$ to $(a,b)$. Then $f_N$ and $g$ are crossed, which contradicts the hypothesis.\\

\begin{claim*}[$(3)\Longrightarrow (1)+(2)$]
{\it For any subgroup $G$ of  $\rm{Homeo}_{+}(\mathbb{R})$ without crossed elements, if $G$ has no infinite tower covering the line, then there exists a
$G$-invariant Radon measure and  a nonempty minimal closed subset.}
\end{claim*}

\noindent\textbf{Case 1} $\text{Fix}(G)\neq \emptyset$. Then take any fixed point $x\in\text{Fix}(G)$, the Dirac measure $\delta_x$ is a $G$ invariant Radon measure and $\{x\}$ is a minimal closed subset.\\

\noindent\textbf{Case 2} $\text{Fix}(G)=\emptyset$. Let $F=\{f\in G: \text{Fix}(f)\neq\emptyset\}$ and let $\Gamma=\langle F\rangle$. By Lemma \ref{infinity tower}
and Lemma \ref{generator fix}, $\text{Fix}(\Gamma)\neq \emptyset$. Hence $\Gamma = F$ and $\Gamma$ is a proper normal subgroup of $G$, by Fact \ref{fact1}. Thus $\text{Fix}(\Gamma)$ is $G$-invariant.\\

\noindent\textbf{Subcase 2a} $\Gamma=\{\text {id}\}$. In this case, the $G$-action is free. By H\"{o}lder's Lemma \ref{Holder}, this action is conjugate to translations on the line.
Note that the Lebesgue measure is translation invariant and there always exists a minimal closed subset $M$ for any group consisting of translations. Pulling back the Lebesgue measure and the minimal subset $M$ by the conjugation, we obtained a $G$-invariant Radon measure and a minimal closed subset.\\

\noindent\textbf{Subcase 2b} $\Gamma$ is nontrivial and $\text{Fix}(\Gamma)$ is uncountable. This case has been proved in Lemma \ref{uncountable case}.\\

\noindent\textbf{Subcase 2c} $\Gamma$ is nontrivial and $\text{Fix}(\Gamma)$ is countable. Choose $g\in G\setminus \Gamma$ and $x_0\in \text{Fix}(\Gamma)$. We may assume that $g(x)>x$ for any $x\in\mathbb{R}$. Set $x_n=g^n(x_0), n\in\mathbb{Z}$.
 Since $\text{Fix}(g)=\emptyset$, $\{x_n: n\in\mathbb{Z}\}$ has no accumulating points. Set
 \begin{displaymath}
 X=[x_{-1},x_2],~~Y=\text{Fix}(\Gamma)\cap [x_{-1},x_2].
 \end{displaymath}
 Then  $Y$ and $Y\cap [x_{0},x_1] $ are countable compact nonempty subsets of $X$.\\

Define $Y_0$ to be the set of isolated points in $Y$, which is nonempty since $Y$ is countable and compact. Moreover, $Y_0\cap [x_0,x_1]$ is nonempty. Set $Y_1=Y\setminus Y_0$ which is a proper closed subset of $Y$. Define $Y_2=Y_1\setminus\{$isolated points in $Y_1$ under the subspace topology$\}$.
For an ordinal $\beta$, suppose that we have defined the nonempty closed subsets $Y_{\alpha}$ for all $\alpha<\beta$. If $\beta=\alpha+1$, define $Y_{\beta}=Y_{\alpha}\setminus\{$isolated points in $Y_{\alpha}$ under subspace topology$\}$. If $\beta$ is a limit ordinal, then define $Y_{\beta}=\bigcap_{\alpha<\beta} Y_{\alpha}$, which is nonempty by compactness. Since $Y$ is countable, there must exist a countable ordinal $\gamma$ such that
\begin{displaymath}
Y_{\gamma} \cap [x_{0},x_1]\neq \emptyset,~{\rm{and}}~Y_{\gamma+1}\cap [x_{0},x_1]=\emptyset.
\end{displaymath}
Thus every point of $Y_{\gamma}\cap [x_{0},x_1]$ is isolated in $Y_{\gamma}$ under the subspace topology. \\

Take $y\in Y_{\gamma}\cap [x_{0},x_1]$. We claim that $Gy$ is a closed subset of $\mathbb{R}$ without accumulating points.
Otherwise, there exist $f_n\in G$ such that $f_n(y)\rightarrow z\in \text{Fix}(\Gamma)$ as $n\rightarrow\infty$. Let $k\in \mathbb{Z}$ be such that $z\in [x_{k},x_{k+1}]$. Then $g^{-k}f_n(y)\in Y_{\gamma}\cap [x_{0},x_1]$ for sufficiently large $n$ and $g^{-k}f_n(y)\rightarrow g^{-k}(z)\in [x_0,x_1]$ as $n\rightarrow\infty$. Then $Y_{\gamma}$ has an accumulating point in $[x_0,x_1]$ which is a contradiction (note that for any $\alpha\leq\gamma$, $x\in Y_{\alpha}\cap [x_{0},x_1]$ and $f\in G$, if $f(x)\in [x_{0},x_1]$ then $f(x)\in Y_{\alpha}\cap [x_{0},x_1]$, since $f$ is a homeomorphism).  Thus $Gy$ is a discrete sequence $(y_n)_{n\in\mathbb{Z}}$ which is unbounded in both directions. Let $\mu=\sum_{n\in \mathbb{Z}}\delta_{y_n}$. Then $\mu$ is a $G$-invariant Radon measure and $Gy$ is a minimal closed subset.

\section{proof of Theorem \ref{main corollary}}
For a nilpotent subgroup of ${\rm{Homeo}}_{+}(\mathbb{R})$, if it does not have an invariant Radon measure, then we can construct an infinite tower which is available for us to deal with the smooth case. Precisely, we have the following lemma. (We use $\mathbb{N}^{+}$ to denote the set of positive integers.)\\

\begin{lemma}\label{tower for nilpotent}
Let $G$ be a nilpotent subgroup of ${\rm{Homeo}}_{+}(\mathbb{R})$. If there does not exist $G$-invariant Radon measure of $\mathbb{R}$, then there exist subgroups $A,B$ of $G$, a closed interval $I_0$ and an infinite tower $(I_i, h_i)_{i=1}^{\infty}$ such that
\begin{itemize}
  \item [(1)] for any $i\in\mathbb{N}^{+}$, $I_i$ is a closed interval and $I_i$ is contained in the interior of $I_{i+1}$;
  \item [(2)]  $\forall j\in\mathbb{N}^{+}$, ${\rm{Fix}}(h_j)\cap I_{j}={\rm{End}}(I_j)$;
  \item [(3)] $I_0\subseteq {\rm{int}} (I_1)$ and $ {\rm{Fix}}(A)\cap I_{0}={\rm{End}}(I_0)$;
  \item [(4)] $A\vartriangleleft B, [B,B]\leq A$, and $h_j\in B$, $\forall j\in\mathbb{N}^{+}$.
\end{itemize}
\end{lemma}

\begin{proof}
Let $H$ be a subgroup of $G$ generated by the elements that have fixed points. Then, by Lemma \ref{generator fix}, every element of $H$ has fixed points and $H$ is a normal subgroup of $G$, by Fact \ref{fact1}.\\

\noindent\textbf{Claim 1}. $H\neq \{e\}$.\\
\indent If $H= \{e\}$, then the action of $G$ is free. Thus, by H\"{o}lder's Theorem \ref{Holder}, $G$ consists of the translations of the line. Then the Lebesgure measure is an invariant Radon measure, which contracts the hypothesis of the lemma.\\

\noindent\textbf{Claim 2}. ${\rm{Fix}}(H)= \emptyset$.\\
\indent If ${\rm{Fix}}(H)\neq \emptyset$, then we conclude that $G$ have an invariant Radon measure and a minimal subset by Lemma \ref{uncountable case} for case that ${\rm{Fix}}(H)$ is uncountable and by Subcase 2c in the proof of Theorem \ref{main theorem 1} for the case that ${\rm{Fix}}(H)$ is countable. Thus we get a contraction again. \\

Since $H$ is nilpotent, there exist a finite normal series $\{e\}=H_0\vartriangleleft H_1\vartriangleleft\cdots\vartriangleleft H_n=H$ of $H$, for some positive integer $n$, such that $[H, H_{i+1}]\leq H_i$, for any $i=1,\cdots,n-1$. Since ${\rm{Fix}}(H)=\emptyset$ by Claim 2, we can take $m\in\{1,\cdots,n\}$ to be the least integer such that ${\rm{Fix}}(H_m)=\emptyset$.\\

\noindent\textbf{Case 1}. $m=1$.\\
\indent In this case, take a nontrivial element $h_0\in H_1$. Then take a connected component $(a_0,b_0)$ of $\mathbb{R}\setminus {\rm{Fix}}(h_0)$. We claim that $a_0,b_0\in\mathbb{R}$. In fact, at least one of $a_0,b_0$ is finite, since $h_0$ is nontrivial. We may assume that $a_0\in\mathbb{R}$. By the assumption that ${\rm{Fix}}(H_1)=\emptyset$, there exists some $h\in H_1$ such that $h(a_0)>a_0$. Note that $H_1$ is commutative. Thus we have $h(a_0)\in {\rm{Fix}}(h_0)$, and then $b_0\leq h(a_0)$. Therefore, $a_0,b_0\in\mathbb{R}$ and ${\rm{Fix}}(h_0)\cap [a_0,b_0]=\{a_0,b_0\}$.\\
\indent Take $h_1\in H_1$ such $h_1(b_0)>b_0$. Then $h_1((a_0,b_0))\cap (a_0,b_0)=\emptyset$, by the commutativity of $H_1$. Thus ${\rm{Fix}}(h_1)\cap [a_0,b_0]=\emptyset$. Let $(a_1,b_1)$ be the connected component of $\mathbb{R}\setminus {\rm{Fix}}(h_1)$ containing $[a_0,b_0]$. By the similar arguments as above, we have $a_1,b_1\in\mathbb{R}$ and ${\rm{Fix}}(h_1)\cap [a_1,b_1]=\{a_1,b_1\}$. Proceeding in this way, we obtain an infinite tower $([a_i,b_i],h_i)_{i=1}^\infty$ in the end.\\
\indent Take $A=\langle h_0\rangle$, $B=H_1$, $I_i=[a_i, b_i]$, and $h_i$ to be as above. Then $A$, $B$, $I_0$ and $(I_i, h_i)_{i=1}^\infty$ such defined satisfy the requirements.\\

\noindent\textbf{Case 2}. $m>1$.\\
\indent In this case, we take $A=H_{m-1}$ and $B=H_{m}$. Take a connected component $(a_0,b_0)$ of $\mathbb{R}\setminus {\rm{Fix}}(A)$. By similar arguments as above, we have $a_0,b_0\in\mathbb{R}$ and ${\rm{Fix}}(A)\cap [a_0,b_0]=\{a_0,b_0\}$ (Note that ${\rm{Fix}}(A)$ is $B$-invariant, since $A$ is normal in $B$). \\
\indent Since ${\rm{Fix}}(B)=\emptyset$, there exists $h_1\in B$ such that $h_1(b_0)>b_0$, which implies that $h_1(a_0, b_0)\cap (a_0, b_0)=\emptyset$. Thus ${\rm{Fix}}(h_1)\cap [a_0,b_0]=\emptyset$.  Let $(a_1,b_1)$ be the connected component of $\mathbb{R}\setminus {\rm{Fix}}(h_1)$ containing $[a_0,b_0]$. Then $a_1,b_1\in\mathbb{R}$ and ${\rm{Fix}}(h_1)\cap [a_1,b_1]=\{a_1,b_1\}$. Moreover, $a_1,b_1\in {\rm{Fix}}(A)$, since $h_1^i(a_0),h_1^i(b_0)\in {\rm{Fix}}(A)$ for all $i$ and $\lim_{i\rightarrow -\infty}h_1^{i}(a_0)=a_1, \lim_{i\rightarrow +\infty}h_1^{i}(b_0)=b_1$.\\
\indent Now $b_1\in {\rm{Fix}}(A)\cap {\rm{Fix}}(h_1)$. Then we can take $h_2\in B$ such $h_2(b_1)>b_1$ by ${\rm Fix}(B)=\emptyset$. Similarly, we can take an interval $[a_2,b_2]$ such that $[a_1,b_1]\subseteq (a_2,b_2)$ and ${\rm{Fix}}(h_2)\cap [a_2,b_2]=\{a_2,b_2\}$. Since $[B, B]\subset A$, the group $\langle A, h_1\rangle$ 
is normal in $B$. Then we have further $\{a_2, b_2\}\subset {\rm Fix}(A)\cap {\rm Fix}(h_1)\cap {\rm Fix}(h_2)$. 
 Inductively, we can obtain an infinite tower $([a_i,b_i],h_i)_{i=1}^\infty$ which satisfies
\begin{itemize}
  \item [(i)] $\forall i\in \mathbb{N}^{+}, [a_i,b_i]\subseteq(a_{i+1},b_{i+1})$,
  \item [(ii)] $\forall i\in\mathbb{N}^{+}$, $h_i\in B$, and ${\rm{Fix}}(h_i)\cap [a_i,b_i]=\{a_i,b_i\}$,
  \item [(iii)] $\forall i\in\mathbb{N}^{+}$, $\{a_{i},b_{i}\}\subset {\rm{Fix}}(A)\cap {\rm{Fix}}(\langle h_1,\cdots,h_i\rangle)$.
\end{itemize}
Thus we complete the proof by taking $A=H_{m-1}$, $B=H_m$ $I_i=[a_i,b_i]$,  and $h_i$ to be as above.\\
\end{proof}

To prove Theorem \ref{main corollary}, we need the following version of generalised Kopell Lemma.

\begin{lemma}[\cite{Na2} Proposition 2.8]\label{kopell}
Given an integer $k\geq 3$, let $\{L_{l_1,\cdots,l_k}: (l_1,\cdots,l_k)\in\mathbb{Z}^k\}$ be a family of closed intervals with disjoint interiors and disposed on $[0,1]$ respecting the lexicographic order, that is, $L_{l_1,\cdots,l_k}$ is to the left of $L_{l'_1,\cdots,l'_k}$ if and only if $(l_1,\cdots,l_k)$ is lexicographically smaller that $(l'_1,\cdots,l'_k)$. Let $h_1,\cdots,h_k$ be $C^{1}$ diffeomorphisms of $[0,1]$ such that for each $j\in\{1,\cdots,k\}$ and each $(l_1,\cdots,l_k)\in\mathbb{Z}^k$ one has
\[
h_j(L_{l_1,\cdots,l_{j-1},l_{j},\cdots,l_k})=
L_{l_1,\cdots,l_{j-1},l'_{j},\cdots,l'_k},
\]
for some $(l'_j,l'_{j+1},\cdots, l'_k)\in \mathbb{Z}^{k-j+1}$ satisfying $l'_j\neq l_j$. If $\alpha>0$ satisfies
\[
\alpha(1+\alpha)^{k-2}\geq 1,
\]
then $h_1,\cdots,h_{k-1}$ cannot be simultaneously contained in $\rm{Diff}^{1+\alpha}_{+}([0,1])$.
\end{lemma}

\begin{proof}[Proof of Theorem \ref{main corollary}]
Let $G$ be a nilpotent subgroup of ${\rm{Diff}}_{+}^{1+\alpha}(\mathbb{R})$, for some $\alpha>0$. By Theorem \ref{main theorem 1}, it suffices to show that there exists an invariant Radon measure.\\

To the contrary, if there does not exist any invariant Radon measure, then, by Lemma \ref{tower for nilpotent}, there exist subgroups $A,B$ of $G$, a closed interval $I_0$ and an infinite tower $(I_i, h_i)_{i=1}^{\infty}$ satisfying the properties $(1)-(4)$ in Lemma \ref{tower for nilpotent}. Moreover, we may assume that $h_i(x)>x$, for any $i\in\mathbb{N}^{+}$ and any $x\in I_0$; otherwise, we can replace it by its inverse. Take a positive integer $k\geq 3$ such that $\alpha(1+\alpha)^{k-2}\geq 1$ and set
\[ \mathcal{L}=\{g(I_0): g\in \langle A\cup\{h_1,\cdots, h_{k}\}\rangle\}.\]
\textbf{Claim}. For each $L\in \mathcal{L}$, $L$ is contained in the interior of $I_{k+1}$ and there exists a unique $(l_1,\cdots,l_k)\in\mathbb{Z}^{k}$ such that $L=h_{1}^{l_1}\cdots h_i^{l_{i}}\cdots h_{k}^{l_k}(I_0)$.\\
\indent In fact, we set $\Gamma=\langle A\cup\{h_1,\cdots, h_{k}\}\rangle$. Since $\Gamma\leq B$ and $[B,B]\leq A$, we have that $A\vartriangleleft \Gamma$ and $\Gamma/ A$ is an Abelian group with finite rank. Note that ${\rm{Homeo}}_{+}(\mathbb{R})$ is torsion-free. Thus, for any $g\in \Gamma$, there exists a unique $(l_1,\cdots,l_k)\in\mathbb{Z}^{k}$ such that $gA=h_{1}^{l_1}\cdots h_{k}^{l_k}A$. Hence the second part of the claim holds by the fact that $I_0$ is $A$-invariant. The first part is easily followed from the observation that the end points of $I_{k}$ are contained in ${\rm{Fix}}(\Gamma)$. Thus the claim holds.\\

For every $(l_1,\cdots,l_k)\in\mathbb{Z}^{k}$,  it is clear that
\begin{eqnarray*}
h_{1}^{l_1}\cdots h_{k}^{l_k}(I_0)&\subseteq& h_{2}^{l_2}\cdots h_{k}^{l_k}(I_1)\\&\subseteq &\cdots\subseteq h_i^{l_{i}}\cdots h_{k}^{l_k}(I_{i-1})\subseteq\cdots\\&\subseteq& h_{k}^{l_k}(I_{k-1})\subseteq I_{k}.
\end{eqnarray*}
Set $L_{l_1,\cdots,l_k}=h_{1}^{l_1}\cdots h_{k}^{l_k}(I_0)$. Thus $\{L_{l_1,\cdots,l_k}: (l_1,\cdots,l_k)\in\mathbb{Z}^k\}$ is a family of closed intervals contained in $I_{k}$ with disjoint interiors. By the assumption that $h_i(x)>x$, for every $i\geq 2$ and every $x\in I_0$, $\{L_{l_1,\cdots,l_k}: (l_1,\cdots,l_k)\in\mathbb{Z}^k\}$ are disposed on $I_{k+1}$ respecting the lexicographic order.
By the Claim, we have $\mathcal{L}=\{L_{l_1,\cdots,l_k}: (l_1,\cdots,l_k)\in\mathbb{Z}^k\}$. Now for each $j\in\{2,\cdots,k+1\}$ and each $(l_1,\cdots,l_k)\in\mathbb{Z}^k$ one has
\[
h_j(L_{l_1,\cdots,l_{j-1},l_{j},l_{j+1},\cdots,l_k})=
L_{l_1,\cdots,l_{j-1},l_{j}+1,l_{j+1},\cdots,l_k}.
\]
By the choice of $k$ and Lemma \ref{kopell}, we know that $h_1,\cdots,h_{k-1}$ cannot be contained in ${\rm{Diff}}^{1+\alpha}_{+}(I_{k})$ simultaneously, which contradicts the hypothesis that $G$ is a subgroup of ${\rm{Diff}}_{+}^{1+\alpha}(\mathbb{R})$. This completes the proof.
\end{proof}

\section{A Counterexample of $C^1$ Subgroup}

In this section, we construct an example which shows that Theorem \ref{main corollary} does not hold for $C^1$ commutative subgroups of  $\text{Homeo}_+(\mathbb{R})$. The following construction is due to Yoccoz(\cite[Lemma 2.1]{Farb}).

\begin{lemma}\label{equivariant family}
For any closed intervals $I=[a,b] ,J=[c,d]$ there exists a $C^1$ orientation preserving diffeomorphism $\phi_{I,J}: I\longrightarrow J$ with the following properties:
\begin{enumerate}
  \item [(1)]  $\phi_{I,J}'(a)=\phi_{I,J}'(b)=1$;
  \item [(2)]  Given $\varepsilon>0$, there exists $\delta>0$ such that for all $x\in [a,b]$,
  \begin{displaymath}
  \left| \phi_{I,J}'(x)-1 \right|<\varepsilon,~~\text{whenever}~~\left|\frac{d-c}{b-a}-1\right|<\delta;
  \end{displaymath}
  \item [(3)] For any closed interval $K$ and for any $x\in I$,
  \begin{displaymath}
  \phi_{I,K}(x)=\phi_{J,K}(\phi_{I,J}(x)).
  \end{displaymath}
\end{enumerate}
\end{lemma}

\begin{theorem}
There exists a non-finitely generated abelian group $G$ consisting of $C^1$ orientation preserving diffeomorphisms of $\mathbb{R}$ such that there exists an infinite tower $\{(I_j,f_j)\}_{j=1}^{\infty}$ with $f_j\in G, j=1,2,\cdots$, such that $\bigcup_{j=1}^{\infty}I_j=\mathbb{R}$. \\
(Then, by Theorem \ref{main theorem 1}, there exists neither $G$ invariant Radon measure nor nonempty closed minimal set.)
\end{theorem}
\begin{proof}
Firstly, we define $f_1: [-1,1]\longrightarrow [-1,1]$ by
\begin{equation*}
f_1(x)=
\begin{cases}
\exp\left(\frac{1}{x-1}-\frac{1}{x+1}\right)+x,& x\in(-1,1)\\
-1,&x=-1\\
1,& x=1.
\end{cases}
\end{equation*}
Then $f_1$ satisfies
\begin{itemize}
  \item $f_1$ is a $C^1$ orientation preserving diffeomorphism of $[-1,1]$;
  \item $f_1(\pm1)=\pm 1$ and $f_1(x)> x $ for any $x\in (-1,1)$;
  \item $f_1'(-1)=f_1'(1)=1$.
\end{itemize}
Next, choose two infinite sequences $-2<\cdots<a_2<a_1<a_0=-1$ and $1=b_0<b_1<b_2<\cdots<2$ such that
\begin{displaymath}
\lim_{n\rightarrow \infty}a_n=-2,~~\lim_{n\rightarrow\infty} b_n=2,
\end{displaymath}
and
\begin{displaymath}
 \lim_{n\rightarrow\infty} \frac{a_{n-1}-a_{n}}{a_n-a_{n+1}} =1,~\lim_{n\rightarrow\infty} \frac{b_{n+1}-b_{n}}{b_n-b_{n-1}} =1.
\end{displaymath}
For example, we can take
\begin{displaymath}
a_n= -2+\frac{1}{n+1},~b_n=2-\frac{1}{n+1},~n=1,2,\cdots.
\end{displaymath}
Define
\begin{equation*}
f_2(x)=
\begin{cases}
\phi_{[a_{n+1},a_n],[a_n,a_{n-1}]}(x),& x\in [a_{n+1},a_n],n=1,2,\cdots\\
\phi_{[a_{1},a_0],[-1,1]}(x),& x\in [a_{1},a_0]\\
\phi_{[-1,1],[b_0,b_1],}(x),& x\in [-1,1]\\
\phi_{[b_{n},b_{n+1}],[b_{n+1},b_{n+2}]}(x),& x\in [b_n,b_{n+1}],n=0,1,2,\cdots\\
\pm2,& x=\pm2.
\end{cases}
\end{equation*}
Then, by Lemma\ref{equivariant family} and the choices of $\{a_n\}$ and $\{b_n\}$, $f_2$ satisfies
\begin{itemize}
  \item $f_2$ is a $C^1$ orientation preserving diffeomorphism of $[-2,2]$;
  \item $f_2(\pm2)=\pm 2$ and $f_2(x)> x $ for any $x\in (-2,2)$;
  \item $f_2'(-2)=f_2'(2)=1$.
\end{itemize}
Then we extend $f_1$ to a diffeomorphism $\tilde{f_1}$ of $[-2,2]$:
\begin{equation*}
\tilde{f_1}(x)=
\begin{cases}
f_2^{-(n+1)}f_1f_2^{n+1}(x),& x\in [a_{n+1},a_n],n=1,2,\cdots\\
f_1(x),& x\in[-1,1]\\
f_2^{n+1}f_1f_2^{-(n+1)}(x),& x\in [b_n,b_{n+1}]\\
\pm2,& x=\pm2.
\end{cases}
\end{equation*}
We denote $\tilde{f_1}$  by $f_1$  for $x\in[-2,2]$. Then
\begin{displaymath}
f_1f_2(x)=f_2f_1(x),~~\forall x\in [-2,2].
\end{displaymath}
Continuing the above process, we can construct a sequence of commuting  $C^1$ orientation preserving diffeomorphisms $f_1,f_2,\cdots$ of $\mathbb{R}$. More precisely, assume that we have constructed pairwise commuting $C^1$ orientation preserving diffeomorphisms $f_1,\cdots,f_k$ of $[-k,k]$ for $k\in\mathbb{N}^+$ with the following properties:
\begin{enumerate}
  \item $f_i(\pm i)=\pm i$ and $\forall x\in (-i,i)$, $f_{i}(x)>x$, for $i=1,2,\cdots,k$;
  \item $f_i'(- i)=f_i'(i)=1$, for $i=1,2,\cdots,k$;
  \item $f_if_j(x)=f_jf_i(x)$ for all $x\in[-k,k]$ and $1\leq i,j\leq k$.
\end{enumerate}
Then choose two infinite sequences $-(k+1)<\cdots<c_2<c_1<c_0=-k$ and $k=d_0<d_1<d_2<\cdots<k+1$ such that
\begin{displaymath}
\lim_{n\rightarrow \infty}c_n=-(k+1),~~\lim_{n\rightarrow\infty} d_n=k+1,
\end{displaymath}
and
\begin{displaymath}
 \lim_{n\rightarrow\infty} \frac{c_{n-1}-c_{n}}{c_n-c_{n+1}} =1,~\lim_{n\rightarrow\infty} \frac{d_{n+1}-d_{n}}{d_n-d_{n-1}} =1.
\end{displaymath}
For example, we can take
\begin{displaymath}
c_n= -(k+1)+\frac{1}{n+1},~d_n=k+1-\frac{1}{n+1},~n=1,2,\cdots.
\end{displaymath}
Define
\begin{equation*}
f_{k+1}(x)=
\begin{cases}
\phi_{[c_{n+1},c_n],[c_n,c_{n-1}]}(x),& x\in [c_{n+1},c_n],n=1,2,\cdots\\
\phi_{[c_{1},c_0],[-k,k]}(x),& x\in [c_{1},c_0]\\
\phi_{[-k,k],[d_0,d_1],}(x),& x\in [-k,k]\\
\phi_{[d_{n},d_{n+1}],[d_{n+1},d_{n+2}]}(x),& x\in [d_n,d_{n+1}],n=0,1,2,\cdots\\
\pm(k+1),& x=\pm(k+1).
\end{cases}
\end{equation*}
Then  by Lemma \ref{equivariant family} and the choices of $\{c_n\}$ and $\{d_n\}$, $f_{k+1}$ satisfies
\begin{itemize}
  \item $f_{k+1}$ is a $C^1$ orientation preserving diffeomorphism of $[-k-1,k+1]$;
  \item $f_{k+1}(\pm(k+1))=\pm (k+1)$ and $f_{k+1}(x)> x $ for any $x\in (-(k+1),k+1)$;
  \item $f_{k+1}'(-k-1)=f_{k+1}'(k+1)=1$.
\end{itemize}
We extend $f_1,\cdots,f_k$ to diffeomorphisms $\tilde{f_1}, \cdots,\tilde{f_k}$ of $[-(k+1),k+1]$: for $i=1,\cdots,k$,
\begin{equation*}
\tilde{f_i}(x)=
\begin{cases}
f_{k+1}^{-(n+1)}f_{i}f_{k+1}^{n+1}(x),& x\in [c_{n+1},c_n],n=1,2,\cdots\\
f_k(x),& x\in[-k,k],\\
f_{k+1}^{n+1}f_if_{k+1}^{-(n+1)}(x),& x\in [d_n,d_{n+1}],\\
\pm(k+1),& x=\pm(k+1).
\end{cases}
\end{equation*}
Denote $\tilde{f_i}$ by $f_i$ for $x\in [-(k+1),k+1]$. Then $f_1,\cdots,f_{k+1}$ are commuting orientation preserving $C^1$ diffeomorphisms of $[-(k+1),k+1]$.

From the constructing process, we see that $[-1,1]\subsetneq [-2,2]\subsetneq\cdots$ and $f_1,f_2,\cdots$ form an infinite tower, and the group $G$ generated by $f_1,f_2,\cdots$ is a non-finitely generated abelian group consisting of $C^1$ orientation preserving diffeomorphisms of $\mathbb{R}$. This completes the proof.
\end{proof}
It was pointed out by the referee that this example is essentially in \cite{BF}. However, it is worthwhile to construct it explicitly here,
especially for the convenience of the readers.
\vspace{5mm}

\subsection*{Acknowledgements}
We want to express our sincere thanks to the referees for valuable and thoughtful comments,  which highly improved the original version of the paper. \\

The work is supported by NSFC (No. 11771318, No. 11790274).

\vspace{5mm}

\noindent Enhui Shi, School of mathematical and sciences, Soochow University, Suzhou, 215006, P.R. China (ehshi@suda.edu.cn)\\

\noindent Yiruo Wang, School of mathematical and sciences, Soochow University, Suzhou, 215006, P.R. China (632804847@qq.com)\\

\noindent Hui Xu, School of mathematical and sciences, Soochow University, Suzhou, 215006, P.R. China (20184007001@stu.suda.edu.cn)

\end{document}